\documentclass[12pt,reqno]{amsart}

\usepackage{amsmath,amssymb,amsthm,mathrsfs}
\usepackage{graphicx,cite,times}
\usepackage{cases}
\usepackage{color}
\usepackage{appendix}
\usepackage{enumerate}

\usepackage[colorlinks, linkcolor=blue, citecolor=blue]{hyperref}
\usepackage{cite}

\newtheorem{theo}{Theorem}[section]
\newtheorem{coro}{Corollary}[section]
\newtheorem{lemm}{Lemma}[section]

\newtheorem{rema}{Remark}[section]

\newtheorem{exam}{Example}[section]

\numberwithin{equation}{section}

\begin{document}

\title[Inverse problem for the plate equation]{On an inverse problem for the plate equation with passive measurement}

\author{Yixian Gao}
\address{School of
Mathematics and Statistics, Center for Mathematics and
Interdisciplinary Sciences, Northeast Normal University, Changchun, Jilin 130024, China}
\email{gaoyx643@nenu.edu.cn}

\author{Hongyu Liu}
\address{Department of Mathematics, City University of Hong Kong, Kowloon, Hong Kong, China }
\email{hongyu.liuip@gmail.com, hongyliu@cityu.edu.hk}

\author{Yang Liu}
\address{School of Mathematics and Statistics,  Northeast Normal University, Changchun,
Jilin 130024, China}
\email{liuy694@nenu.edu.cn}

\thanks{The research of YG was  supported by NSFC grants 11871140, 12071065 and National Key R\&D Program of China 2020YFA0714102.
 The research
of HL was supported by Hong Kong RGC General Research Funds (project numbers,
11300821, 12301218 and 12302919) and the NSFC-RGC Joint Research Grant (project number,
 N\_CityU101/21).
}


\begin{abstract}
This paper focuses on an inverse problem associated with the plate equation which is derived from models in fluid mechanics and elasticity.
We establish the unique identifying results in simultaneously determining both the unknown density and internal sources from passive boundary measurement.
The proof mainly relies on the asymptotic analysis and harmonic analysis on integral transforms.
\end{abstract}

\maketitle

\section{introduction}\label{se1}
Let $\Omega$ be a compact set in $ \mathbb R^3$    such that $\mathbb R^3\backslash\Omega$ is connected.
Consider the  following  plate equation
\begin{align}\label{plate equation}
\begin{cases}
	\rho( \boldsymbol x)\partial_t^2 u(t, \boldsymbol x)+\Delta^2 u(t,  \boldsymbol x)=0,~\quad (t,  \boldsymbol x)\in \mathbb R_+\times \mathbb R^3,\\
	u(0, \boldsymbol x)=f(\boldsymbol x),\quad\partial_t u(0,  \boldsymbol x)=g(\boldsymbol x),\quad  \boldsymbol x\in \mathbb R^3,
\end{cases}
\end{align}
where the sources $f(\boldsymbol x),  g(\boldsymbol x) \in L^\infty(\mathbb R^3)$ with  $\text{supp}(f)\subset\Omega$ and  $\text{supp}(g)\subset\Omega$ and
  the density function $\rho( \boldsymbol x)\in L^\infty(\mathbb R^3)$ is nonnegative and $\text{supp}(\rho(\boldsymbol x)-1)\subset \Omega$.
Associated with \eqref{plate equation}, we introduce the boundary measurement as follows,
\begin{align}\label{measure}
	\Lambda_{\rho, f, g}(t, \boldsymbol x)=(u(t,  \boldsymbol x), ~\Delta u(t,  \boldsymbol x)),~\quad (t, \boldsymbol x)\in \mathbb R_+\times \partial\Omega.
\end{align}
The inverse problem that we are concerned with is to determine both the unknown density $\rho(\boldsymbol x)$ and the internal sources $f(\boldsymbol x)$ and $g(\boldsymbol x)$ simultaneously by knowledge of $\Lambda_{\rho, f, g}(t, \boldsymbol x)$, $(t, \boldsymbol x)\in \mathbb R_+\times \partial\Omega$. That is,
\begin{equation}\label{eq:ip1}
\Lambda_{\rho, f, g}(t, \boldsymbol x)\big|_{(t, \boldsymbol x)\in \mathbb{R}_+\times \partial\Omega}\longrightarrow \rho, f, g.
\end{equation}
In the physical setup, $\Lambda_{\rho, f, g}(t, \boldsymbol x)$ is referred to as {\em the passive measurement} , since it is generated by certain unknown sources, namely $f$ and $g$. This is in contrast to  {\em the active measurement} where one actively sends a-priori known probing sources and collects the responses for identification purpose.
Here, we would like to point out that the terminology `passive measurement' has different meanings in other fields, such as electronics and information engineering, which should be clear from the context.

The plate equation is a significant portion of engineering construction that arises from the study of mathematical models in fluid mechanics and elasticity.
A mechanical background of the boundary value problem is the elastic bending problem of thin plates.
The inverse scattering problems have played an important role in diverse scientific areas such as radar, sonar, geophysical exploration, and medical imaging.
There have been some literature concentrated on the inverse scattering problems for biharmonic operators, see \cite{Gunther2012,G2010, M2009}.
Compared with the second order differential equations, as acoustic, elastic, and electromagnetic waves, the research of the inverse scattering problem of biharmonic equation is not as extensive as the results of second order differential equations.
The increase of the order leads to the failure of the methods which work for the second order equations.
A detailed description of the properties of the solution can be found in \cite{ma2009}.

Global uniqueness results of recovering potential function or medium parameters associated with bi-harmonic or poly-harmonic operators by active measurements  can be found in \cite{Gunther2012, Yang2014, Yang2017, Liuboya2020}. To our best knowledge, no uniqueness identifiability result is known in the literature in determining unknown medium parameters associated with the bi-harmonic operator by the passive measurement. Indeed, in our study of \eqref{eq:ip1}, we aim at determining both the medium parameter $\rho$ and the unknown sources $f, g$ from the associated passive measurement. In recent years, simultaneously determining both an unknown source and its surrounding medium by the associated passive measurement has received considerable interest in the literature due to its practical importance in emerging applications. In \cite{Liu2015, Knox2020}, the authors proved the uniqueness in determining both an acoustic density and an internal source for the scalar wave equation by the passive measurement, which arises in thermo- and photo-acoustic tomography. In \cite{Liu2019}, the authors established unique recovery results in simultaneously determining an unknown internal current and its surrounding medium by the passive measurement associated with a Maxwell system, which arises in brain imaging. In \cite{DLL1,DLL2}, similar inverse problems were considered associated with the geo-dynamical system which arises in geomagnetic anomaly detection. We also refer to \cite{LLM, Liu2021, cao2019} for more related studies in different physical and mathematical setups. These results show that it is possible to prove the uniqueness of two or more unknowns simultaneously by the passive measurement.
Motivated by \cite{Liu2019, Liu2021, Liu2015, cao2019, Knox2020}, we are interested in recovering both the density $\rho$ and the sources $f, g$ by the passive boundary measurement for the biharmonic plate equation \eqref{plate equation}.

In this article, we shall make use of the temporal Fourier transform, converting the time-dependent problem \eqref{plate equation} into the frequency domain. To that end, we need to require that the plate equation of \eqref{plate equation} is exponentially decaying in time, which can guarantee the well-posedness of the temporal Fourier transform. In order to appeal for a general study, we shall always assume the exponentially decaying in time for the plate equation. Nevertheless, we would like to emphasise that such a property is satisfied by generic mediums and sources. In what follows, we refer to $(f, g, \rho)$ as admissible if the aforementioned time decaying property is fulfilled. In fact, one can refer to the case of the acoustic wave equation  (cf. Theorem 6.1, page 113, \cite{Initial}), and derive general admissibility conditions by following similar arguments. However, this is not the focus of the current study and our main purpose is to study the related inverse problem.
%In fact, if the plate equation satisfies radiating conditions, the time direction  in \eqref{plate equation} can be replaced by a finite time interval $(0, T_0)$, .
%The radiating waves in general leave the bounded domain $\Omega$ after a finite time $T_0$.
%With $T_0$ depending on the $\rho, f, g$, we call that {\bf configuration} $(\rho, f, g)$ is {\bf  admissible}.
%Our main purpose is to study the inverse problem, so not to discuss the explicit relation between $T_0$ and the configuration $(\rho, f, g)$.
The main methods to obtain the uniqueness results are by performing certain asymptotic analysis in the low frequency regime.
We derive some integral identities involving source functions and density coupling.
Combining those integral identities and using harmonic analysis techniques, the uniqueness results are obtained.
The uniqueness relies on assuming that the density $\rho$ and internal source functions $f, g$ along an arbitrary direction vector $\boldsymbol \iota$  satisfy $\nabla\rho\cdot \boldsymbol \iota=\nabla f\cdot \boldsymbol \iota=\nabla g\cdot \boldsymbol \iota=0$.
Additionally, the above assumption can be replaced by setting the size of the corresponding parameters to obtain a more general uniqueness results.

The paper is organized as follows. In Section \ref{sec2}, we introduce the temporal Fourier transform, converting the time domain problem \eqref{plate equation} into frequency domain and give some notations.
In Section \ref{se3}, we derive  the asymptotic expansions of the solutions with respect to the frequency $\kappa$, and get some integral identities which the source functions and the density are coupled together.
Then the uniqueness results established by a natural admissibility assumption on the parameters.
The more general uniqueness results are  given in Section \ref{s4}.
%------------------------------------------------------------------------------------------------------------------------------------------------------------------------------------------------

\section{Problem Formulation}\label{sec2}
In this section, we introduce the temporal Fourier transform to convert the time domain problem \eqref{plate equation} into the frequency domain and introduce some notations.

Our argument depends on the temporal Fourier transform of the  function $u(t, \boldsymbol x)$ defined by
\begin{align*}%\label{temfou}
	\hat u(\omega, \boldsymbol x):= \frac{1}{2\pi} \int_0^\infty u(t, \boldsymbol x) e^{{\rm i}w t}~{\rm d}t,
~\quad  (\omega, \boldsymbol  x)\in \mathbb R_+\times \mathbb R^3.
\end{align*}

Applying the temporal Fourier transform  to equation \eqref{plate equation} and assume that $\kappa=\omega^{1/2}$,  it yields that   $\hat u(\kappa,  \boldsymbol x)$ satisfies
\begin{align}\label{theq}
\Delta^2\hat u(\kappa, \boldsymbol x)-\kappa^4\rho( \boldsymbol x)\hat u(\kappa, \boldsymbol  x)
=-\frac{{\rm i}\kappa^2}{2\pi}\rho(\boldsymbol x)f( \boldsymbol x)+\frac{1}{2\pi}\rho(\boldsymbol x)g(\boldsymbol x ), \quad (\kappa, \boldsymbol x)\in \mathbb R_+\times \mathbb R^3,
\end{align}
and the boundary measurement \eqref{measure}  in the frequency domain
\begin{align}\label{boundary}
	\hat\Lambda_{\rho, f, g}(\kappa, \boldsymbol x)=\left(\hat u(\kappa, \boldsymbol  x), ~\Delta \hat u(\kappa, \boldsymbol x)\right),~\quad (\kappa, \boldsymbol x)\in \mathbb R_+\times \partial\Omega.
\end{align}

To ensure the well-posedness of the equation \eqref{theq}, we impose an analogue of the Sommerfeld radiation conditions
\begin{align}\label{radi}
\underset{r\rightarrow\infty}{\lim}r \left(\partial_r \hat u(\kappa,  \boldsymbol x)-{\rm i}\kappa\hat u(\kappa, \boldsymbol  x)\right)=0	,~\quad
\underset{r\rightarrow\infty}{\lim}r\left(\partial_{r} (\Delta\hat u(\kappa, \boldsymbol x))-{\rm i}\kappa(\Delta\hat u(\kappa, \boldsymbol x))\right)=0,
\end{align}
 uniformly in all directions $\hat {\boldsymbol x}=\boldsymbol x/|\boldsymbol x|$ with $r=|\boldsymbol x|$ (cf. \cite{Tyni2018}).

One of the key technical ingredients to establish the unique recovery results is first by performing certain asymptotic analysis in the low frequency regime to derive certain integral identities involving the source function and density, which are coupled together.

We introduce some notations.
The fundamental solution for biharmonic operator $\Delta^2-\kappa^4$ in $\mathbb R^3$  is
\begin{align}\label{fundmen}
	G_\kappa(|\boldsymbol x-\boldsymbol y|)=\frac{e^{{\rm i}\kappa|\boldsymbol x-\boldsymbol y|}-e^{-\kappa|\boldsymbol x-\boldsymbol y|}}{8\pi\kappa^2|\boldsymbol x-\boldsymbol y|}~\quad \text{for}~\boldsymbol x\neq \boldsymbol y.
\end{align}
Notice that when $\kappa =0,$ the fundamental solution to $\Delta^2$ is
\begin{align*}
G_0(|\boldsymbol x-\boldsymbol y|)=-\frac{|\boldsymbol x-\boldsymbol y|}{8\pi}	~\quad \text{for}~\boldsymbol x\neq \boldsymbol y,
\end{align*}
and the fundamental solution for $-\Delta$ is
\begin{align*}
g_0(|\boldsymbol x-\boldsymbol y|)=\frac{1}{4\pi|\boldsymbol x-\boldsymbol y|} ~\quad\text{for}~\boldsymbol x\neq \boldsymbol y. 	
\end{align*}
%Define inner product in $L^2(\mathbb R^3)$ as follow
%\begin{align*}
%(u, v)_{L^2(\mathbb R^3)}:=\int_{\mathbb R^3}u \bar v~{\rm d}\boldsymbol x.
%\end{align*}
%-------------------------------------------------------------------------------------------------------------------------------------------------------------------------------------------------

\section{The uniqueness results for density and internal sources} \label{se3}
In this section, we will prove the uniqueness for both the unknown  density  $\rho$  and the internal  sources $f$ and $g$.
%--------------------------------------------------------------------------------------------------

\subsection{Auxiliary results}

Before proving the uniqueness results, we first derive several auxiliary results.

\begin{lemm}\label{aux}
	Let $\hat u(\kappa, \boldsymbol x)\in H^2_{loc}(\mathbb R^3)$ be the solution of \eqref{theq} and \eqref{radi}.
	Then $\hat u(\kappa, \boldsymbol x)$ is uniquely given by the following integral equation
	\begin{align}\label{solution}	
	\hat u(\kappa, \boldsymbol x)=&\kappa^4\int_{\mathbb R^3}	(\rho(\boldsymbol y)-1)\hat u(\kappa, \boldsymbol y)G_\kappa(|\boldsymbol x-\boldsymbol y|)~{\rm d}\boldsymbol y-\frac{{\rm i}\kappa^2}{2\pi}\int_{\mathbb R^3}	\rho(\boldsymbol y)f(\boldsymbol y)G_\kappa(|\boldsymbol x-\boldsymbol y|)~{\rm d}\boldsymbol y \nonumber\\
	&+\frac{1}{2\pi}\int_{\mathbb R^3}	\rho(\boldsymbol y)g(\boldsymbol y)G_\kappa(|\boldsymbol x-\boldsymbol y|)~{\rm d}\boldsymbol y,~\quad \boldsymbol x\in \mathbb R^3.
	\end{align}
	
In addition, if taking the expansion of $e^{{\rm i}\kappa|\boldsymbol x-\boldsymbol y|}-e^{-\kappa|\boldsymbol x-\boldsymbol y|}$ as $\kappa\rightarrow+0$, we have
	\begin{align}\label{order3'}
		\hat u(\kappa, \boldsymbol x)=&\frac{1}{2\pi\kappa}\int_{\mathbb R^3}\frac{({\rm i}+1)}{8\pi}\rho(\boldsymbol y)g(\boldsymbol y)~{\rm d}\boldsymbol y-\frac{1}{2\pi}\int_{\mathbb R^3}\frac{1}{8\pi}\rho(\boldsymbol y)g(\boldsymbol y)|\boldsymbol x-\boldsymbol y|~{\rm d}\boldsymbol y\nonumber\\
		&+\frac{\kappa}{2\pi}\int_{\mathbb R^3}\frac{(1-{\rm i})}{8\pi}\rho(\boldsymbol y)g(\boldsymbol y)\frac{|\boldsymbol x-\boldsymbol y|^2}{3!}~{\rm d}\boldsymbol y\nonumber\\
		&-\frac{\rm i\kappa}{2\pi}\int_{\mathbb R^3}\frac{({\rm i}+1)}{8\pi}\rho(\boldsymbol y)f(\boldsymbol y)~{\rm d}\boldsymbol y+\mathcal{O}(\kappa^2),~\quad \boldsymbol x\in B_R,
\end{align}
where $B_R:=B(0, R)$ is a central ball of radius $R\in \mathbb R_+$ and satisfies $\Omega\subset B_R$.
\end{lemm}
\begin{proof}
From the regularity theorem in \cite{evans}, it is easy to verify that the solution $\hat u\in H_{loc}^4(\mathbb R^3)$.
 By applying the fundamental solution to \eqref{theq}, we obtain a Lippmann-Schwinger integral equation
\begin{align}\label{lse}
\hat u(\kappa, \boldsymbol x)=&\kappa^4\int_{\mathbb R^3}	(\rho(\boldsymbol y)-1)\hat u(\kappa, \boldsymbol y) G_\kappa(|\boldsymbol x-\boldsymbol y|) ~{\rm d}\boldsymbol y\nonumber\\
&-\frac{{\rm i}\kappa^2}{2\pi}\int_{\mathbb R^3}\rho(\boldsymbol y)f(\boldsymbol y)G_k(|\boldsymbol x-\boldsymbol y|)~{\rm d}\boldsymbol y\nonumber\\
&+\frac{1}{2\pi}\int_{\mathbb R^3}\rho(\boldsymbol y)g(\boldsymbol y)G_k(|\boldsymbol x-\boldsymbol y|)~{\rm d}\boldsymbol y.
\end{align}
Since $\rho(x)-1=0$ in $\mathbb R^3\backslash\bar \Omega$ and $\rho\in L^\infty(\mathbb R^3)$, we assume that $\Omega\subset B_R$ and $\mathcal K_{\rho, \kappa}:C(\overline{B_R})\longrightarrow C(\overline{B_R})$ satisfies
\begin{align*}
\mathcal K_{\rho, \kappa}(\hat u)=\kappa^4\int_{B_R}	(\rho(\boldsymbol y)-1)\hat u(\kappa, \boldsymbol y)  G_\kappa(|\boldsymbol x-\boldsymbol y|)~{\rm d}\boldsymbol y.
\end{align*}
Suppose that $M:=\underset{|\boldsymbol x|\leq R}{\sup} |\rho(\boldsymbol x)-1|$ and $\kappa^2< \frac{2}{MR^2}$, we have $\|\mathcal K_{\rho, \kappa}\|_{L^\infty(B_R)}\leq 1$.
Therefore, there exsits a Neumann series
\begin{align*}
(I-\mathcal K_{\rho, \kappa})^{-1}=I+\mathcal K_{\rho, \kappa}+\mathcal K_{\rho, \kappa}^2+\cdots.
\end{align*}
If taking $\kappa\longrightarrow+0$ and replacing $e^{{\rm i}\kappa|\boldsymbol x-\boldsymbol y|}-e^{-\kappa|\boldsymbol x-\boldsymbol y|}$ by the series expansion
\begin{align*}
e^{{\rm i}\kappa|\boldsymbol x-\boldsymbol y|}-e^{-\kappa|\boldsymbol x-\boldsymbol y|}=\kappa({\rm i}+1)|\boldsymbol x-\boldsymbol y|-\kappa^2|\boldsymbol x-\boldsymbol y|^2+\frac{\kappa^3(1-{\rm i})}{3!}|\boldsymbol x-\boldsymbol y|^3+\mathcal{O}(\kappa^5),
\end{align*}
we calculate that
\begin{align}\label{appro}
(I-\mathcal K_{\rho, \kappa})^{-1}=I+\mathcal O(\kappa^3).
\end{align}
Additionally, substituting \eqref{fundmen} into \eqref{lse}, implies
\begin{align*}
\hat u(\kappa, \boldsymbol x)
=&(I-\mathcal K_{\rho, \kappa})^{-1}\big(-\frac{{\rm i}\kappa^2}{2\pi}\int_{\mathbb R^3}\rho(\boldsymbol y)f(\boldsymbol y)\frac{e^{{\rm i}\kappa|\boldsymbol x-\boldsymbol y|}-e^{-\kappa|\boldsymbol x-\boldsymbol y|}}{8\pi\kappa^2|\boldsymbol x-\boldsymbol y|}~{\rm d}\boldsymbol y\\
&+\frac{1}{2\pi}\int_{\mathbb R^3}\rho(\boldsymbol y)g(\boldsymbol y)\frac{e^{{\rm i}\kappa|\boldsymbol x-\boldsymbol y|}-e^{-\kappa|\boldsymbol x-\boldsymbol y|}}{8\pi\kappa^2|\boldsymbol x-\boldsymbol y|}~{\rm d}\boldsymbol y\big)\\
=&\frac{1}{2\pi}\int_{\mathbb R^3}\rho(\boldsymbol y)g(\boldsymbol y)\frac{({\rm i}+1)}{8\pi\kappa}-\rho(\boldsymbol y)g(\boldsymbol y)\frac{|\boldsymbol x-\boldsymbol y|}{8\pi}+\rho(\boldsymbol y)g(\boldsymbol y)\frac{\kappa(1-{\rm i})}{3!}\frac{|\boldsymbol x-\boldsymbol y|^2}{8\pi}~{\rm d}\boldsymbol y\\
&-\frac{{\rm i}\kappa}{2\pi}\int_{\mathbb R^3}\rho(\boldsymbol y)f(\boldsymbol y)\frac{({\rm i}+1)|\boldsymbol x-\boldsymbol y|}{8\pi|\boldsymbol x-\boldsymbol y|}~{\rm d}\boldsymbol y+\mathcal{O}(\kappa^2),~\quad \boldsymbol x\in B_R.
\end{align*}
The proof is completed.
\end{proof}

\begin{lemm}
	Let $\hat u(\kappa, \boldsymbol x)\in H^2_{loc}(\mathbb R^3)$ be the solution of \eqref{theq} and \eqref{radi}.
	Then the integral equation \eqref{solution} can be rewritten as
\begin{align}\label{solution1}	
\hat u(\kappa, \boldsymbol x)=\underset{m=-1}{\overset{3}\sum}M_m(\boldsymbol x)\kappa^m+\mathcal O(\kappa^4)~\quad \text{as}~ \kappa\rightarrow +0,~\quad \boldsymbol x\in B_R,
\end{align}
where
\begin{align*}
M_{-1}:=&\frac{1}{2\pi}\frac{({\rm i}+1)}{8\pi}\int_{\mathbb R^3}\rho(\boldsymbol y)g(\boldsymbol y)~{\rm d}\boldsymbol y,\\
M_0:=&\frac{1}{2\pi}\int_{\mathbb R^3}\rho(\boldsymbol y)g(\boldsymbol y)G_0(|\boldsymbol x-\boldsymbol y|)~{\rm d}\boldsymbol y,\\
M_1:=&-\frac{{\rm i}}{2\pi}\frac{({\rm i}+1)}{8\pi}\int_{\mathbb R^3}\rho(\boldsymbol y)f(\boldsymbol y)~{\rm d}\boldsymbol y+\frac{1}{2\pi}\frac{(1-{\rm i})}{8\pi}\int_{\mathbb R^3}\rho(\boldsymbol y)g(\boldsymbol y)\frac{|\boldsymbol x-\boldsymbol y|^2}{3!}~{\rm d}\boldsymbol y,\\
M_2:=&(\frac{{\rm i}+1}{8\pi})\int_{\mathbb R^3}(\rho(\boldsymbol y)-1)M_{-1}~{\rm d}\boldsymbol y -\frac{{\rm i}}{2\pi}\int_{\mathbb R^3}\rho(\boldsymbol y)f(\boldsymbol y)G_0(|\boldsymbol x-\boldsymbol y|)~{\rm d}\boldsymbol y,\\
M_3:=&\int_{\mathbb R^3}(\rho(\boldsymbol y)-1)M_{-1} ~G_0(|\boldsymbol x-\boldsymbol y|)~{\rm d}\boldsymbol y+\frac{({\rm i}+1)}{8\pi}\int_{\mathbb R^3}(\rho(\boldsymbol y)-1)M_0(\boldsymbol y)~{\rm d}\boldsymbol y\\
&-\frac{{\rm i}}{2\pi}\frac{(1-{\rm i})}{8\pi}\int_{\mathbb R^3}\rho(\boldsymbol y)f(\boldsymbol y)\frac{|\boldsymbol x-\boldsymbol y|^2}{3!}~{\rm d}\boldsymbol y+\frac{1}{2\pi}\frac{({\rm i}+1)}{8\pi}\int_{\mathbb R^3}\rho(\boldsymbol y)g(\boldsymbol y)\frac{|\boldsymbol x-\boldsymbol y|^4}{5!}~{\rm d}\boldsymbol y.
\end{align*}

Taking the laplacian for \eqref{solution1} with respect to $\boldsymbol x$, imply
\begin{align}\label{solution2}
\Delta_{\boldsymbol x}
\hat u(\kappa, \boldsymbol x)
=&-\frac{1}{2\pi}\int_{\mathbb R^3}\rho(\boldsymbol y)g(\boldsymbol y)g_0(|\boldsymbol x-\boldsymbol y|)~{\rm d}\boldsymbol y+\frac{\kappa}{2\pi}\frac{(1-{\rm i})}{8\pi}\int_{\mathbb R^3}\rho(\boldsymbol y)g(\boldsymbol y)~{\rm d}\boldsymbol y\nonumber\\
&+\frac{{\rm i}\kappa^2}{2\pi}\int_{\mathbb R^3}\rho(\boldsymbol y)f(\boldsymbol y)g_0(|\boldsymbol x-\boldsymbol y|)~{\rm d}\boldsymbol y-\kappa^3\int_{\mathbb R^3}(\rho(\boldsymbol y)-1)M_{-1}~g_0(|\boldsymbol x-\boldsymbol y|)~{\rm d}\boldsymbol y\nonumber\\
&-\frac{{\rm i}\kappa^3}{2\pi}\frac{(1-{\rm i})}{8\pi}\int_{\mathbb R^3}\rho(\boldsymbol y)f(\boldsymbol y)~{\rm d}\boldsymbol y\nonumber\\
&+\frac{\kappa^3}{2\pi}\frac{({\rm i}+1)}{8\pi}\int_{\mathbb R^3}\rho(\boldsymbol y)g(\boldsymbol y)\frac{|\boldsymbol x-\boldsymbol y|^2}{3!}~{\rm d}\boldsymbol y+\mathcal O(\kappa^4)~\quad \text{as}~ \kappa\rightarrow +0,~\quad \boldsymbol x\in B_R.
\end{align}
\end{lemm}
\begin{proof}
Plugging \eqref{order3'} into \eqref{solution} and using the series expansion
\begin{align*}
		e^{{\rm i}\kappa|\boldsymbol x-\boldsymbol y|}-e^{-\kappa|\boldsymbol x-\boldsymbol y|}=&\kappa({\rm i}+1)|\boldsymbol x-\boldsymbol y|-\kappa^2|\boldsymbol x-\boldsymbol y|^2+\frac{\kappa^3(1-{\rm i})}{3!}|\boldsymbol x-\boldsymbol y|^3\nonumber\\
		&+\frac{\kappa^5({\rm i}+1)}{5!}|\boldsymbol x-\boldsymbol y|^5+\mathcal{O}(\kappa^6)~\quad \text{as}~ \kappa\rightarrow+0,
\end{align*}
 uniformly for $\boldsymbol  x \in B_R$, we get
\begin{align}\label{aux3}
\hat u(\kappa, \boldsymbol x)=
&\kappa^{4}\int_{\mathbb R^3}(\rho(\boldsymbol y)-1)\big(\frac{1}{2\pi\kappa}\int_{\mathbb R^3}\frac{({\rm i}+1)}{8\pi}\rho(\boldsymbol z)g(\boldsymbol z)~{\rm d}\boldsymbol z-\frac{1}{2\pi}\int_{\mathbb R^3}\rho(\boldsymbol z)g(\boldsymbol z)\frac{|\boldsymbol y-\boldsymbol z|}{8\pi}~{\rm d}\boldsymbol z\nonumber\\
&+\frac{\kappa}{2\pi}\int_{\mathbb R^3}\frac{(1-{\rm i})}{8\pi}\rho(\boldsymbol z)g(\boldsymbol z)\frac{|\boldsymbol y-\boldsymbol z|^2}{3!}~{\rm d}\boldsymbol z\nonumber\\
&-\frac{\rm i\kappa}{2\pi}\int_{\mathbb R^3}\frac{({\rm i}+1)}{8\pi}\rho(\boldsymbol z)f(\boldsymbol z)~{\rm d}\boldsymbol z+\mathcal{O}(\kappa^2)\big)\frac{e^{{\rm i}\kappa|\boldsymbol x-\boldsymbol y|}-e^{-\kappa|\boldsymbol x-\boldsymbol y|}}{8\pi\kappa^2|\boldsymbol x-\boldsymbol y|} ~{\rm d}\boldsymbol y\nonumber\\
&-\frac{{\rm i}\kappa^2}{2\pi}\int_{\mathbb R^3}\rho(\boldsymbol y)f(\boldsymbol y)\frac{e^{{\rm i}\kappa|\boldsymbol x-\boldsymbol y|}-e^{-\kappa|\boldsymbol x-\boldsymbol y|}}{8\pi\kappa^2|\boldsymbol x-\boldsymbol y|}~{\rm d}\boldsymbol y\nonumber\\
&+\frac{1}{2\pi}\int_{\mathbb R^3}\rho(\boldsymbol y)g(\boldsymbol y)\frac{e^{{\rm i}\kappa|\boldsymbol x-\boldsymbol y|}-e^{-\kappa|\boldsymbol x-\boldsymbol y|}}{8\pi\kappa^2|\boldsymbol x-\boldsymbol y|}~{\rm d}\boldsymbol y\nonumber\\
:=&I_1+I_2+I_3+\mathcal O(\kappa^4).
\end{align}
By direct calculation, one has
\begin{align*}
I_1=&\kappa^4\int_{\mathbb R^3}(\rho(\boldsymbol y)-1)\big(\frac{1}{2\pi\kappa}\int_{\mathbb R^3}\frac{({\rm i}+1)}{8\pi}\rho(\boldsymbol z)g(\boldsymbol z)~{\rm d}\boldsymbol z-\frac{1}{2\pi}\int_{\mathbb R^3}\rho(\boldsymbol z)g(\boldsymbol z)\frac{|\boldsymbol y-\boldsymbol z|}{8\pi}~{\rm d}\boldsymbol z\\
&+\frac{\kappa}{2\pi}\int_{\mathbb R^3}\frac{(1-{\rm i})}{8\pi}\rho(\boldsymbol z)g(\boldsymbol z)\frac{|\boldsymbol y-\boldsymbol z|^2}{3!}~{\rm d}\boldsymbol z\\
&-\frac{\rm i\kappa}{2\pi}\int_{\mathbb R^3}\frac{({\rm i}+1)}{8\pi}\rho(\boldsymbol z)f(\boldsymbol z)~{\rm d}\boldsymbol z+\mathcal{O}(\kappa^2)\big)\frac{e^{{\rm i}\kappa|\boldsymbol x-\boldsymbol y|}-e^{-\kappa|\boldsymbol x-\boldsymbol y|}}{8\pi\kappa^2|\boldsymbol x-\boldsymbol y|} ~{\rm d}\boldsymbol y\\
=&\kappa^2\int_{\mathbb R^3}(\rho(\boldsymbol y)-1)\frac{1}{2\pi}\int_{\mathbb R^3}\frac{({\rm i}+1)}{8\pi}\rho(\boldsymbol z)g(\boldsymbol z)~{\rm d}\boldsymbol z~\frac{({\rm i}+1)}{8\pi}~{\rm d}\boldsymbol y\\
&+\kappa^3\int_{\mathbb R^3}(\rho(\boldsymbol y)-1)\frac{1}{2\pi}\int_{\mathbb R^3}\frac{({\rm i}+1)}{8\pi}\rho(\boldsymbol z)g(\boldsymbol z)~{\rm d}\boldsymbol z~G_0(|\boldsymbol x-\boldsymbol y|)~{\rm d}\boldsymbol y\\
&+\kappa^3\int_{\mathbb R^3}(\rho(\boldsymbol y)-1)\frac{1}{2\pi}\int_{\mathbb R^3}\rho(\boldsymbol z)g(\boldsymbol z)G_0(|\boldsymbol y-\boldsymbol z|)~{\rm d}\boldsymbol z~\frac{({\rm i}+1)}{8\pi}~ {\rm d}\boldsymbol y+\mathcal{O}(\kappa^4),~\quad \boldsymbol x\in B_R,
\end{align*}
\begin{align*}
I_2=&-\frac{{\rm i}\kappa^2}{2\pi}\int_{\mathbb R^3}\rho(\boldsymbol y)f(\boldsymbol y)\frac{e^{{\rm i}\kappa|\boldsymbol x-\boldsymbol y|}-e^{-\kappa|\boldsymbol x-\boldsymbol y|}}{8\pi\kappa^2|\boldsymbol x-\boldsymbol y|}~{\rm d}\boldsymbol y\\
=&-\frac{{\rm i}\kappa}{2\pi}\frac{({\rm i}+1)}{8\pi}\int_{\mathbb R^3}\rho(\boldsymbol y)f(\boldsymbol y)~{\rm d}\boldsymbol y-\frac{{\rm i}\kappa^2}{2\pi}\int_{\mathbb R^3}\rho(\boldsymbol y)f(\boldsymbol y)G_0(|\boldsymbol x-\boldsymbol y|)~{\rm d}\boldsymbol y\\
&-\frac{{\rm i}\kappa^3}{2\pi}\frac{(1-{\rm i})}{8\pi}\int_{\mathbb R^3}\rho(\boldsymbol y)f(\boldsymbol y)\frac{|\boldsymbol x-\boldsymbol y|^2}{3!}~{\rm d}\boldsymbol y+\mathcal{O}(\kappa^4),~\quad \boldsymbol x\in B_R,
\end{align*}
and
\begin{align*}
I_3=&\frac{1}{2\pi}\int_{\mathbb R^3}\rho(\boldsymbol y)g(\boldsymbol y)\frac{e^{{\rm i}\kappa|\boldsymbol x-\boldsymbol y|}-e^{-\kappa|\boldsymbol x-\boldsymbol y|}}{8\pi\kappa^2|\boldsymbol x-\boldsymbol y|}~{\rm d}\boldsymbol y\\
=&\frac{1}{2\pi\kappa}\frac{({\rm i}+1)}{8\pi}\int_{\mathbb R^3}\rho(\boldsymbol y)g(\boldsymbol y)~{\rm d}\boldsymbol y+\frac{1}{2\pi}\int_{\mathbb R^3}\rho(\boldsymbol y)g(\boldsymbol y)G_0(|\boldsymbol x-\boldsymbol y|)~{\rm d}\boldsymbol y\nonumber\\
&+\frac{\kappa}{2\pi}\frac{(1-{\rm i})}{8\pi}\int_{\mathbb R^3}\rho(\boldsymbol y)g(\boldsymbol y)\frac{|\boldsymbol x-\boldsymbol y|^2}{3!}~{\rm d}\boldsymbol y\\
&+\frac{\kappa^3}{2\pi}\frac{({\rm i}+1)}{8\pi}\int_{\mathbb R^3}\rho(\boldsymbol y)g(\boldsymbol y)\frac{|\boldsymbol x-\boldsymbol y|^4}{5!}~{\rm d}\boldsymbol y+\mathcal O(\kappa^4),~\quad \boldsymbol x\in B_R.
\end{align*}
Taking the Laplacian on both side of equality \eqref{aux3} with respect to $\boldsymbol x$, we obtain
\begin{align*}
\Delta_{\boldsymbol x}\hat u(\kappa, \boldsymbol x)
=&-\frac{\kappa^3}{2\pi}\frac{({\rm i}+1)}{8\pi}\int_{\mathbb R^3}(\rho(\boldsymbol y)-1)\int_{\mathbb R^3}\rho(\boldsymbol z)g(\boldsymbol z)~{\rm d} \boldsymbol z~g_0(|\boldsymbol x-\boldsymbol y|)~{\rm d}\boldsymbol y\\
&+\frac{{\rm i}\kappa^2}{2\pi}\int_{\mathbb R^3}\rho(\boldsymbol y)f(\boldsymbol y)g_0(|\boldsymbol x-\boldsymbol y|)~{\rm d}\boldsymbol y-\frac{{\rm i}\kappa^3}{2\pi}\frac{(1-{\rm i})}{8\pi}\int_{\mathbb R^3}\rho(\boldsymbol y)f(\boldsymbol y)~{\rm d}\boldsymbol y\\
&-\frac{1}{2\pi}\int_{\mathbb R^3}\rho(\boldsymbol y)g(\boldsymbol y)g_0(|\boldsymbol x-\boldsymbol y|)~{\rm d}\boldsymbol y+\frac{\kappa}{2\pi}\frac{(1-{\rm i})}{8\pi}\int_{\mathbb R^3}\rho(\boldsymbol y)g(\boldsymbol y)~{\rm d}\boldsymbol y\\
&+\frac{\kappa^3}{2\pi}\frac{({\rm i}+1)}{8\pi}\int_{\mathbb R^3}\rho(\boldsymbol y)g(\boldsymbol y)\frac{|\boldsymbol x-\boldsymbol y|^2}{3!}~{\rm d}\boldsymbol y+\mathcal O(\kappa^4),~\quad \boldsymbol x\in B_R
\end{align*}
as $\kappa\rightarrow +0$.
The proof is completed.
\end{proof}

\begin{rema}
Let $n\in \mathbb N\cup\{0\}$, then the solution $\hat u(\kappa, \boldsymbol x)$ can be  represented as
\begin{align*}
	\hat u(\kappa, \boldsymbol x)=\underset{m=-1}{\overset{n+1}\sum}M_m(\boldsymbol x)\kappa^m+M_{n+2}(\boldsymbol x)\kappa^{n+2}+\mathcal O(\kappa^{n+3}),
\end{align*}
and
\begin{align*}
	\Delta_{\boldsymbol x}\hat u(\kappa, \boldsymbol x)=\underset{m=0}{\overset{n+2}\sum}N_m(\boldsymbol x)\kappa^m+N_{n+3}(\boldsymbol x)\kappa^{n+3}+\mathcal O(\kappa^{n+4})~\quad \text{as}~ \kappa\rightarrow +0,
\end{align*}
 where
\begin{align*}
M_{n+2}(\boldsymbol x):=&\underset{m=0}{\overset{n}\sum}\int_{\mathbb R^3}(\rho(\boldsymbol y)-1)M_{n-m-1}(\boldsymbol y)\frac{{\rm i}^{m+1}-(-1)^{m+1}}{8\pi}\frac{|\boldsymbol x-\boldsymbol y|^m}{(m+1)!}~{\rm d}\boldsymbol y\\
&-\frac{{\rm i}}{2\pi}\int_{\mathbb R^3}\rho(\boldsymbol y)f(\boldsymbol y)\frac{{\rm i}^{n+2}-(-1)^{n+2}}{8\pi}\frac{|\boldsymbol x-\boldsymbol y|^{n+1}}{(n+2)!}~{\rm d}\boldsymbol y\\
&+\frac{1}{2\pi}\int_{\mathbb R^3}\rho(\boldsymbol y)g(\boldsymbol y)\frac{{\rm i}^{n+4}-(-1)^{n+4}}{8\pi}\frac{|\boldsymbol x-\boldsymbol y|^{n+3}}{(n+4)!}~{\rm d}\boldsymbol y,
\end{align*}
and
\begin{align*}
N_{n+3}:=&\underset{m=1}{\overset{n+1}\sum}\int_{\mathbb R^3}(\rho(\boldsymbol y)-1)M_{n-m}(\boldsymbol y)\frac{{\rm i}^{m+1}-(-1)^{m+1}}{8\pi}\frac{|\boldsymbol x-\boldsymbol y|^{m-2}}{(m-1)!}~{\rm d}\boldsymbol y\\
&-\frac{{\rm i}}{2\pi}\int_{\mathbb R^3}\rho(\boldsymbol y)f(\boldsymbol y)\frac{{\rm i}^{n+3}-(-1)^{n+3}}{8\pi}\frac{|\boldsymbol x-\boldsymbol y|^{n}}{(n+1)!}~{\rm d}\boldsymbol y\\
&+\frac{1}{2\pi}\int_{\mathbb R^3}\rho(\boldsymbol y)g(\boldsymbol y)\frac{{\rm i}^{n+5}-(-1)^{n+5}}{8\pi}\frac{|\boldsymbol x-\boldsymbol y|^{n+2}}{(n+3)!}~{\rm d}\boldsymbol y, ~\quad ~\boldsymbol x\in B_R.
\end{align*}
In addition, $N_m(\boldsymbol x)\kappa^m, m=0, 1, 2$ corresponds to the first three terms in \eqref{solution2}, respectively.
\end{rema}

\begin{theo}\label{t1}
Assume that  $(\rho_1, f_1, g_1)$ and $(\rho_2, f_2, g_2)$ are two sets of admissible configurations and supported in $\Omega$.
If
	\begin{align}\label{equlity}
	\Lambda_{\rho_1, f_1, g_1}(t, \boldsymbol x)=\Lambda_{\rho_2, f_2, g_2}(t, \boldsymbol x),~\quad (t, \boldsymbol x)\in \mathbb R_+\times \partial\Omega.
	\end{align}
Then for any harmonic function $h(\boldsymbol x)$, we have
\begin{align}
\int_{\mathbb R^3}(\rho_1f_1-\rho_2f_2)(\boldsymbol x)h(\boldsymbol x)	~{\rm d}\boldsymbol x=0,\label{uniqueness}\\
\int_{\mathbb R^3}(\rho_1g_1-\rho_2g_2)(\boldsymbol x) h(\boldsymbol x)	~{\rm d}\boldsymbol x=0.\label{uniqueness1}
\end{align}
Furthermore, for any $\boldsymbol x\in \partial B_R$, the following holds
\begin{align}\label{eq3}
&\int_{\mathbb R^3}(\rho_1(\boldsymbol y)-1)\int_{\mathbb R^3}\rho_1(\boldsymbol z)g_1(\boldsymbol z)~{\rm d}\boldsymbol z ~g_0(|\boldsymbol x-\boldsymbol y|)~{\rm d}\boldsymbol y
\nonumber\\
&\hspace{3cm}+\int_{\mathbb R^3}\rho_1(\boldsymbol y)f_1(\boldsymbol y)~{\rm d}\boldsymbol y-\int_{\mathbb R^3}\rho_1(\boldsymbol y)g_1(\boldsymbol y)\frac{|\boldsymbol x-\boldsymbol y|^2}{3!}~{\rm d}\boldsymbol y\nonumber\\
=&\int_{\mathbb R^3}(\rho_2(\boldsymbol y)-1)\int_{\mathbb R^3}\rho_2(\boldsymbol z)g_2(\boldsymbol z)~{\rm d}\boldsymbol z ~g_0(|\boldsymbol x-\boldsymbol y|)~{\rm d}\boldsymbol y\nonumber\\
&\hspace{3cm}+\int_{\mathbb R^3}\rho_2(\boldsymbol y)f_2(\boldsymbol y)~{\rm d}\boldsymbol y-\int_{\mathbb R^3}\rho_2(\boldsymbol y)g_2(\boldsymbol y)\frac{|\boldsymbol x-\boldsymbol y|^2}{3!}~{\rm d}\boldsymbol y.
\end{align}
\end{theo}
\begin{proof}
Using the temporal Fourier transform,  let  $\hat u_1(\kappa, \boldsymbol x)$ and  $\hat u_2(\kappa, \boldsymbol x)$ denote  the solution of \eqref{theq}, corresponding to $(\rho_1, f_1, g_1)$ and $(\rho_2, f_2, g_2)$ respectively.
It follows from \eqref{boundary} and \eqref{equlity} that
\begin{align*}
(\hat u_1(\boldsymbol x), ~\Delta\hat u_1(\boldsymbol x))=(\hat u_2(\boldsymbol x)	,~\Delta\hat u_2(\boldsymbol x)),~\quad \boldsymbol x\in \partial\Omega.
\end{align*}
Since the priori information for density and internal sources, it is easy to verify that both $\hat u_1$ and $\hat u_2$ satisfy the same  equation
\begin{align*}
\Delta^2 u-\kappa^4u=0~\quad \text{in}~\mathbb R^3\backslash\overline\Omega.
\end{align*}
Additionally, from the uniqueness of exterior boundary value problem in Theorem \ref{direct}, we have
\begin{align*}
(\hat u_1(\boldsymbol x), ~\Delta\hat u_1(\boldsymbol x))=(\hat u_2(\boldsymbol x)	,~\Delta\hat u_2(\boldsymbol x)),~\quad \boldsymbol x\in \mathbb R^3\backslash\Omega.
\end{align*}
Due to $\partial B_R\subset \mathbb R^3\backslash\overline\Omega$, we can  obtain
\begin{align}\label{equali1}
(\hat u_1(\boldsymbol x), ~\Delta\hat u_1(\boldsymbol x))=(\hat u_2(\boldsymbol x)	,~\Delta\hat u_2(\boldsymbol x)),~\quad \boldsymbol x\in\partial B_R.
\end{align}
Combining  \eqref{solution2} and \eqref{equali1}, we imply the following integral identities
\begin{align}
\frac{1}{2\pi}\int_{B_R}\rho_1(\boldsymbol y)g_1(\boldsymbol y)g_0(|\boldsymbol x-\boldsymbol y|)~{\rm d}\boldsymbol y=&\frac{1}{2\pi}\int_{B_R}\rho_2(\boldsymbol y)g_2(\boldsymbol y)g_0(|\boldsymbol x-\boldsymbol y|)~{\rm d}\boldsymbol y, \label{order1}\\
\frac{{\rm i}\kappa^2}{2\pi}\int_{B_R}\rho_1(\boldsymbol y)f_1(\boldsymbol y)g_0(|\boldsymbol x-\boldsymbol y|)~{\rm d}\boldsymbol y=&\frac{{\rm i}\kappa^2}{2\pi}\int_{B_R}\rho_2(\boldsymbol y)f_2(\boldsymbol y)g_0(|\boldsymbol x-\boldsymbol y|)~{\rm d}\boldsymbol y,\label{order2}
\end{align}
and
\begin{align*}
&\frac{\kappa^3}{2\pi}\frac{({\rm i}+1)}{8\pi}\big(\int_{B_R}(\rho_1(\boldsymbol y)-1)\int_{B_R}\rho_1(\boldsymbol z)g_1(\boldsymbol z)~{\rm d} \boldsymbol z~ g_0(|\boldsymbol x-\boldsymbol y|)~{\rm d}\boldsymbol y\nonumber\\
&\hspace{3cm}+\int_{B_R}\rho_1(\boldsymbol y)f_1(\boldsymbol y)~{\rm d}\boldsymbol y-\int_{B_R}\rho_1(\boldsymbol y)g_1(\boldsymbol y)\frac{|\boldsymbol x-\boldsymbol y|^2}{3!}~{\rm d}\boldsymbol y\big)\nonumber\\
=&\frac{\kappa^3}{2\pi}\frac{({\rm i}+1)}{8\pi}\big(\int_{B_R}(\rho_2(\boldsymbol y)-1)\int_{B_R}\rho_2(\boldsymbol z)g_2(\boldsymbol z)~{\rm d} \boldsymbol z~ g_0(|\boldsymbol x-\boldsymbol y|)~{\rm d}\boldsymbol y\nonumber\\
&\hspace{3cm}+\int_{B_R}\rho_2(\boldsymbol y)f_2(\boldsymbol y)~{\rm d}\boldsymbol y-\int_{B_R}\rho_2(\boldsymbol y)g_2(\boldsymbol y)\frac{|\boldsymbol x-\boldsymbol y|^2}{3!}~{\rm d}\boldsymbol y\big)
\end{align*}
for $\boldsymbol x\in \partial B_R$.

Note that the fundamental solution of $-\Delta$ can be written as
\begin{align}\label{spher}
\frac{1}{4\pi|\boldsymbol x-\boldsymbol y|}=\underset{m=0}{\overset{\infty}{\sum}}\underset{n=-m}{\overset{m}{\sum}}\frac{1}{2m+1}\frac{|\boldsymbol y|^m}{|\boldsymbol x|^{m+1}}Y^n_m(\frac{\boldsymbol x}{|\boldsymbol x|})\overline{Y}^n_m(\frac{\boldsymbol y}{|\boldsymbol y|})~\quad \text{for}~|\boldsymbol x|>|\boldsymbol y|,
\end{align}
where $Y^n_m(\cdot)$ denotes the spherical harmonics of order $m\in \mathbb N\cup \{0\}$ for $n=-m, \cdots, m$.

Substituting \eqref{spher} into \eqref{order1} and \eqref{order2}, we calculate that
\begin{align*}
	&\int_{B_R}\underset{m=0}{\overset{\infty}{\sum}}\underset{n=-m}{\overset{m}{\sum}}\frac{1}{2m+1}\frac{1}{|\boldsymbol x|^{m+1}}Y^n_m(\frac{\boldsymbol x}{|\boldsymbol x|})\rho_1(\boldsymbol y)f_1(\boldsymbol y)|\boldsymbol y|^m\overline{Y}^n_m(\frac{\boldsymbol y}{|\boldsymbol y|})~{\rm d}\boldsymbol y\\
	=&\int_{B_R}\underset{m=0}{\overset{\infty}{\sum}}\underset{n=-m}{\overset{m}{\sum}}\frac{1}{2m+1}\frac{1}{|\boldsymbol x|^{m+1}}Y^n_m(\frac{\boldsymbol x}{|\boldsymbol x|})\rho_2(\boldsymbol y)f_2(\boldsymbol y)|\boldsymbol y|^m\overline{Y}^n_m(\frac{\boldsymbol y}{|\boldsymbol y|})~{\rm d}\boldsymbol y \quad \text{for}~\boldsymbol x\in \partial B_R,
\end{align*}
and
\begin{align*}
	&\int_{B_R}\underset{m=0}{\overset{\infty}{\sum}}\underset{n=-m}{\overset{m}{\sum}}\frac{1}{2m+1}\frac{1}{|\boldsymbol x|^{m+1}}Y^n_m(\frac{\boldsymbol x}{|\boldsymbol x|})\rho_1(\boldsymbol y)g_1(\boldsymbol y)|\boldsymbol y|^m\overline{Y}^n_m(\frac{\boldsymbol y}{|\boldsymbol y|})~{\rm d}\boldsymbol y\\
	=&\int_{B_R}\underset{m=0}{\overset{\infty}{\sum}}\underset{n=-m}{\overset{m}{\sum}}\frac{1}{2m+1}\frac{1}{|\boldsymbol x|^{m+1}}Y^n_m(\frac{\boldsymbol x}{|\boldsymbol x|})\rho_2(\boldsymbol y)g_2(\boldsymbol y)|\boldsymbol y|^m\overline{Y}^n_m(\frac{\boldsymbol y}{|\boldsymbol y|})~{\rm d}\boldsymbol y \quad \text{for}~\boldsymbol x\in \partial B_R.
\end{align*}
That is,
\begin{align*}
&\underset{m=0}{\overset{\infty}{\sum}}\underset{n=-m}{\overset{m}{\sum}}\frac{1}{2m+1}\frac{1}{|R|^{m+1}}Y^n_m(\frac{\boldsymbol x}{|\boldsymbol x|})\int_{B_R}\rho_1(\boldsymbol y)f_1(\boldsymbol y)|\boldsymbol y|^m\overline{Y}^n_m(\frac{\boldsymbol y}{|\boldsymbol y|})~{\rm d}\boldsymbol y\\
	=&\underset{m=0}{\overset{\infty}{\sum}}\underset{n=-m}{\overset{m}{\sum}}\frac{1}{2m+1}\frac{1}{|R|^{m+1}}Y^n_m(\frac{\boldsymbol x}{|\boldsymbol x|})\int_{B_R}\rho_2(\boldsymbol y)f_2(\boldsymbol y)|\boldsymbol y|^m\overline{Y}^n_m(\frac{\boldsymbol y}{|\boldsymbol y|})~{\rm d}\boldsymbol y \quad \text{for}~\boldsymbol x\in \partial B_R,
	\end{align*}
	and
\begin{align*}
&\underset{m=0}{\overset{\infty}{\sum}}\underset{n=-m}{\overset{m}{\sum}}\frac{1}{2m+1}\frac{1}{|R|^{m+1}}Y^n_m(\frac{\boldsymbol x}{|\boldsymbol x|})\int_{B_R}\rho_1(\boldsymbol y)g_1(\boldsymbol y)|\boldsymbol y|^m\overline{Y}^n_m(\frac{\boldsymbol y}{|\boldsymbol y|})~{\rm d}\boldsymbol y\\
	=&\underset{m=0}{\overset{\infty}{\sum}}\underset{n=-m}{\overset{m}{\sum}}\frac{1}{2m+1}\frac{1}{|R|^{m+1}}Y^n_m(\frac{\boldsymbol x}{|\boldsymbol x|})\int_{B_R}\rho_2(\boldsymbol y)g_2(\boldsymbol y)|\boldsymbol y|^m\overline{Y}^n_m(\frac{\boldsymbol y}{|\boldsymbol y|})~{\rm d}\boldsymbol y \quad \text{for}~\boldsymbol x\in \partial B_R.
	\end{align*}	
Indeed,  $\{Y^n_m(\cdot)\}_{m=0, 1, 2, ....; n=-m, ..., m}$ is a complete orthonormal basis of $L^2(\mathbb S^2)$ (note that $\mathbb S^2$ means the unit sphere), we get
\begin{align*}
\int_{B_R}\big(\rho_1(\boldsymbol y)f_1(\boldsymbol y)-\rho_2(\boldsymbol y)f_2(\boldsymbol y)\big)|\boldsymbol y|^m\overline{Y}^n_m(\frac{\boldsymbol y}{|\boldsymbol y|})~{\rm d}\boldsymbol y
	=0,
\end{align*}
and
\begin{align*}
	\int_{B_R}\big(\rho_1(\boldsymbol y)g_1(\boldsymbol y)-\rho_2(\boldsymbol y)g_2(\boldsymbol y)\big)|\boldsymbol y|^m\overline{Y}^n_m(\frac{\boldsymbol y}{|\boldsymbol y|})~{\rm d}\boldsymbol y
	=0,
\end{align*}
where $m\in \mathbb N\cup \{0\}$ for $n=-m, ..., m$.

For any homogeneous harmonic function $h(\cdot)$ can be represent by $|\boldsymbol y|^mY^n_m(\frac{\boldsymbol y}{|\boldsymbol y|})$ for $m=0, 1, 2, ...$, and $n=-m, ..., m$, which yields
\begin{equation}\label{eq:i1}
\begin{split}
	\int_{B_R}(\rho_1f_1-\rho_2f_2)(\boldsymbol y)h(\boldsymbol y)~{\rm d}\boldsymbol y=&0, \quad \boldsymbol x\in \partial B_R,\\
\int_{B_R}(\rho_1g_1-\rho_2g_2)(\boldsymbol y)h(\boldsymbol y)~{\rm d}\boldsymbol y=&0, \quad \boldsymbol x\in \partial B_R.
\end{split}
\end{equation}
\end{proof}

\begin{rema}
Since $f_i, g_i$ are supported in $\Omega$ and $\text{supp}(\rho_i-1)\subset \Omega, i=1, 2$, we note that the integral domains in \eqref{eq:i1} can be replaced by $\Omega$.
\end{rema}
%----------------------------------------------------------------------------------------------------------

\subsection{The uniqueness results}
\begin{theo}\label{t2}
Assume that  $(\rho_1, f_1, g_1)$ and $(\rho_2, f_2, g_2)$ are two sets of admissible configurations and supported in $\Omega$, respectively.
Furthermore, suppose that
	\begin{align*}
		F(\boldsymbol x):=&(\rho_1f_1-\rho_2f_2)(\boldsymbol x),~\quad G(\boldsymbol x):=(\rho_1g_1-\rho_2g_2)(\boldsymbol x), ~\quad \boldsymbol x\in \Omega,
	\end{align*}
%If  $F(\boldsymbol x)$ and  $G(\boldsymbol x)$
 satisfy the either of the following conditions:
\begin{enumerate}[(\rm i)]
	\item $F(\boldsymbol x)=h_1(\boldsymbol x)$ and $ G(\boldsymbol x)=h_2(\boldsymbol x)$ for  $\boldsymbol x\in\Omega$, where $h_1(\boldsymbol x)$ and $ h_2(\boldsymbol x)$ are harmonic functions in $\mathbb R^3$;
	\item $\nabla F(\boldsymbol x)\cdot \boldsymbol\iota=0$ and $\nabla G(\boldsymbol x)\cdot\boldsymbol\iota=0$, where $\boldsymbol \iota$ is an arbitrary direction vector in $\mathbb R^3$.
\end{enumerate}
Then
\begin{align*}
F(\boldsymbol x)= G(\boldsymbol x)=0~\quad \text{for}~a. e. ~~\boldsymbol x\in\Omega.	
\end{align*}
\end{theo}
\begin{proof}
For the first case, taking $h(\boldsymbol x)= h_1(\boldsymbol x)$ and $h(\boldsymbol x)= h_2(\boldsymbol x)$ into \eqref{uniqueness} and \eqref{uniqueness1}, respectively,
we get
\begin{align*}
\int_\Omega h_1^2(\boldsymbol x)~{\rm d}\boldsymbol x=\int_\Omega h_2^2(\boldsymbol x)~{\rm d}\boldsymbol x=0.
\end{align*}
It shows that $F(\boldsymbol x)=G(\boldsymbol x)= 0$.

For the second case, because of the rotation invariance of the biharmonic operator $\Delta^2$, the vector $\boldsymbol\iota$ can be rotated appropriately to any coordinate axis.
Without loss of generality, we assume that $\boldsymbol \iota$ rotates to the $x_3$-axis, then we have
\begin{align*}
\partial_{x_3}F(\boldsymbol x)=	\partial_{x_3}G(\boldsymbol x)=0,
\end{align*}
which  means $F(\boldsymbol x), G(\boldsymbol x)$ only depend on the variables $x_1, x_2$ for $(x_1, x_2, x_3)\in \mathbb R^3$.

Denote by
\begin{align}\label{harmoeq}
	h(\boldsymbol x)=e^{{\rm i}\tilde{\boldsymbol\xi}\cdot \boldsymbol x},\quad \boldsymbol x\in \mathbb R^3
\end{align}
be a harmonic function, where
\begin{align*}
\tilde{\boldsymbol {\xi}}=\boldsymbol \xi_1+{\rm i}\boldsymbol \xi_2,~\quad \boldsymbol \xi_1=(\xi_1, \xi_2, 0)^\top\in \mathbb R^3, ~\quad \boldsymbol\xi_2=(0, 0, \xi_3)^\top\in \mathbb R^3,
\end{align*}
and satisfies $\xi_1^2+\xi_2^2=\xi_3^2$.

Plugging \eqref{harmoeq} into \eqref{uniqueness}, we compute
\begin{align*}
\int_{\mathbb R^3}F(x_1, x_2)e^{{\rm i}\tilde{\boldsymbol \xi}\cdot \boldsymbol x}~{\rm d}\boldsymbol x
=&\int_{B_R}F(x_1, x_2)e^{{\rm i}\tilde{\boldsymbol \xi}\cdot \boldsymbol x}~{\rm d}\boldsymbol x\\
=&\int_{\mathbb R^2}F(x_1, x_2)e^{{\rm i}\xi_1\cdot x_1+{\rm i}\xi_2\cdot x_2}~{\rm d}x_1{\rm d}x_2\int_{\{x_3; (x_1, x_2, x_3)\in B_R\}} e^{-\xi_3x_3}{\rm d}x_3
=0.
\end{align*}
It follows from the priori information of $\rho_i, f_i$ and $g_i, i=1, 2$ that
\begin{align*}
	0=\int_{\mathbb R^2}F(x_1, x_2)e^{{\rm i}\xi_1\cdot x_1+{\rm i}\xi_2\cdot x_2}~{\rm d}x_1{\rm d}x_2=(\mathcal F F)(\boldsymbol \xi_1),
\end{align*}
which holds for any $(\xi_1, \xi_2)\in \mathbb R^2$.
Since $(\mathcal F F)(\boldsymbol \xi_1)$ means the Fourier transform of $F(x_1, x_2)$, it clearly showed that $F(\boldsymbol x)=0$.
 We can state $G(\boldsymbol x)=0$ by using the same methods.
The proof is completed.
\end{proof}

Now, we discuss the uniqueness for density and internal sources by using above orthogonality results.
\begin{theo}\label{t3}
Assume that  $(\rho_1, f_1, g_1)$ and $(\rho_2, f_2, g_2)$ are two sets of configurations and supported in $\Omega$, which satisfy both conditions:
\begin{enumerate}[(i)]
	\item $\rho_1$ and $\rho_2$ are positive constants;
	\item $\nabla f_i(\boldsymbol x)\cdot \boldsymbol\iota=0$ and $\nabla g_i(\boldsymbol x)\cdot\boldsymbol\iota=0, i=1, 2$, where $\boldsymbol \iota$ is an arbitrary direction vector in $\mathbb R^3$.
\end{enumerate}
If
	\begin{align}\label{boun22}
	\Lambda_{\rho_1, f_1, g_1}(t, \boldsymbol x)=\Lambda_{\rho_2, f_2, g_2}(t, \boldsymbol x),~\quad (t, \boldsymbol x)\in \mathbb R_+\times \partial\Omega,
	\end{align}
and suppose that
\begin{align}\label{additioncon}
\int_{\Omega}g_i(\boldsymbol x)~{\rm d}\boldsymbol x\neq 0, ~\quad i=1, 2,~\quad\boldsymbol x\in \Omega.
\end{align}
Then
\[\rho_1=\rho_2,~\quad f_1(\boldsymbol x)=f_2(\boldsymbol x),\quad g_1(\boldsymbol x)=g_2(\boldsymbol x).\]
\end{theo}
\begin{proof}
By rotation invariance, without loss of generality, we still assume that $\boldsymbol \iota$ rotates to the $x_3$-axis.
 Then  $\rho_1f_1(\boldsymbol x)-\rho_2f_2(\boldsymbol x)$  and $\rho_1g_1(\boldsymbol x)-\rho_2g_2(\boldsymbol x)$ only depend on the variables $x_1, x_2$.
By Theorem \ref{t2} and \eqref{boun22},
we deduce
\begin{align*}
	\rho_1f_1(\boldsymbol x)=\rho_2f_2(\boldsymbol x),~\quad \rho_1g_1(\boldsymbol x)=\rho_2g_2(\boldsymbol x),~\quad \boldsymbol x\in \Omega.
\end{align*}
It can be seen that
\begin{align}
\int_{B_R}\rho_1f_1(\boldsymbol y)~{\rm d}\boldsymbol y=&\int_{B_R}\rho_2f_2(\boldsymbol y)~{\rm d}\boldsymbol y,\label{eq4}\\
\int_{B_R}\rho_1g_1(\boldsymbol y)~{\rm d}\boldsymbol y=&\int_{B_R}\rho_2g_2(\boldsymbol y)~{\rm d}\boldsymbol y,\label{eq4'}
\end{align}
and
\begin{align}\label{eq5}
\int_{B_R}\rho_1g_1(\boldsymbol y)|\boldsymbol x-\boldsymbol y|^2~{\rm d}\boldsymbol y=&\int_{B_R}\rho_2g_2(\boldsymbol y)|\boldsymbol x-\boldsymbol y|^2~{\rm d}\boldsymbol y.
\end{align}
Substituting \eqref{eq4}--\eqref{eq5} into \eqref{eq3}, imply
\begin{align*}
(\rho_1-\rho_2)\int_{B_R}\big(\int_{B_R}\rho_1g_1(\boldsymbol z)~{\rm d}\boldsymbol z \big)g_0(|\boldsymbol x-\boldsymbol y|)~{\rm d}\boldsymbol y=0, ~\boldsymbol x\in \partial B_R.
\end{align*}
From the assumption \eqref{additioncon}, we have
\begin{align*}
\rho_1\int_{B_R}g_1(\boldsymbol z)~{\rm d}\boldsymbol z\neq 0.
\end{align*}
It is easy to verify that
\begin{align*}
\int_{B_R}\big(\int_{B_R}\rho_1g_1(\boldsymbol z)~{\rm d}\boldsymbol z \big)g_0(|\boldsymbol x-\boldsymbol y|)~{\rm d}\boldsymbol y\neq 0 ~\quad\text{for} ~\boldsymbol x\in \partial B_R.
\end{align*}
Therefore, we have $\rho_1=\rho_2$ and imply $f_1(\boldsymbol x)=f_2(\boldsymbol x)$ and $g_1(\boldsymbol x)=g_2(\boldsymbol x)$, respectively.
\end{proof}

Next, we prove the uniqueness for density and internal sources in a domain with an anomalous inclusion.
Let $\rho_0(\boldsymbol x)$ be a positive background density which is known in advance and $\varrho_i, i=1, 2$ be a positive constant denoting different anomalous inclusion supported in $\Omega_{0}\subset \Omega$.
\begin{coro}\label{coro1}
Assume that  $(\rho_1, f_1, g_1)$ and $(\rho_2, f_2, g_2)$ are two sets of configurations and supported in $\Omega$, which satisfy	the conditions
\begin{enumerate}[(i)]
	\item $\rho_i(\boldsymbol x)=\rho_0(\boldsymbol x)+\varrho_i\chi_{\Omega_0}$ with $\varrho_i, i=1, 2$ is a constant;
	\item $\nabla\rho_0(\boldsymbol x)\cdot\boldsymbol \iota=0$, $\nabla f_i(\boldsymbol x)\cdot \boldsymbol\iota=0$ and $\nabla g_i(\boldsymbol x)\cdot\boldsymbol\iota=0, i=1, 2$, where $\boldsymbol \iota$ is an arbitrary direction vector in $\mathbb R^3$.
\end{enumerate}
If
	\begin{align*}
	\Lambda_{\rho_1, f_1, g_1}(t, \boldsymbol x)=\Lambda_{\rho_2, f_2, g_2}(t, \boldsymbol x),~\quad (t, \boldsymbol x)\in \mathbb R_+\times \partial\Omega,
	\end{align*}
and  suppose that
\begin{align}\label{addition2}
\int_{\Omega_0}\big(\int_{B_R}\rho_1(\boldsymbol z)g_1(\boldsymbol z)~{\rm d}\boldsymbol z \big)h(\boldsymbol x)~{\rm d}\boldsymbol x\neq0, ~\quad \boldsymbol x\in \Omega_0
\end{align}
for any harmonic function $h(\boldsymbol x)$.
Then
\[\varrho_1=\varrho_2, \quad f_1(\boldsymbol x)=f_2(\boldsymbol x),~\quad g_1(\boldsymbol x)=g_2(\boldsymbol x).\]	
\end{coro}

\begin{proof}
Using analogue analysis as Theorem \ref{t3}.
We assume that $\rho_0$, $f_i$ and $g_i$ only depend on the variables $x_1, x_2$, and write them as $\rho_0(\boldsymbol x)=\rho_0(x_1, x_2), f_i(\boldsymbol x)=f_i(x_1, x_2), g_i(\boldsymbol x)=g_i(x_1, x_2)$ for $(x_1, x_2)\in \mathbb R^2, i=1, 2$,
then we deduce
\begin{align*}
	\rho_1(\boldsymbol x)f_1(\boldsymbol x)=&\rho_2(\boldsymbol x)f_2(\boldsymbol x),~\quad\rho_1(\boldsymbol x)g_1(\boldsymbol x)=\rho_2(\boldsymbol x)g_2(\boldsymbol x),~\quad \boldsymbol x\in \Omega.
\end{align*}
It is easy to see that
\begin{align*}
\int_{B_R}\rho_1(\boldsymbol y)f_1(\boldsymbol y)~{\rm d}\boldsymbol y=&\int_{B_R}\rho_2(\boldsymbol y)f_2(\boldsymbol y)~{\rm d}\boldsymbol y,\\
\int_{B_R}\rho_1(\boldsymbol y)g_1(\boldsymbol y)~{\rm d}\boldsymbol y=&\int_{B_R}\rho_2(\boldsymbol y)g_2(\boldsymbol y)~{\rm d}\boldsymbol y,
\end{align*}
and
\begin{align*}
\int_{B_R}\rho_1(\boldsymbol y)g_1(\boldsymbol y)|\boldsymbol x-\boldsymbol y|^2~{\rm d}\boldsymbol y=&\int_{B_R}\rho_2(\boldsymbol y)g_2(\boldsymbol y)|\boldsymbol x-\boldsymbol y|^2~{\rm d}\boldsymbol y~\quad \text{for}~ \boldsymbol x\in \partial B_R.
\end{align*}
Taking above identities into \eqref{eq3}, imply
\begin{align*}
	&\int_{B_R}(\rho_1(\boldsymbol y)-1)\big(\int_{B_R}\rho_1(\boldsymbol z)g_1(\boldsymbol z)~{\rm d}\boldsymbol z \big)g_0(|\boldsymbol x-\boldsymbol y|)~{\rm d}\boldsymbol y\\
	=&\int_{B_R}(\rho_2(\boldsymbol y)-1)\big(\int_{B_R}\rho_2(\boldsymbol z)g_2(\boldsymbol z)~{\rm d}\boldsymbol z \big)g_0(|\boldsymbol x-\boldsymbol y|)~{\rm d}\boldsymbol y~\quad \text{for}~ \boldsymbol x\in \partial B_R.
\end{align*}
It follows from the proof of Theorem \ref{t1} that
\begin{align}\label{eq6}
	&\int_{B_R}(\rho_1-\rho_2)(\boldsymbol y)\big(\int_{B_R}\rho_1(\boldsymbol z)g_1(\boldsymbol z)~{\rm d}\boldsymbol z \big)h(\boldsymbol y)~{\rm d}\boldsymbol y=0~\quad \text{for}~ \boldsymbol x\in \partial B_R,
\end{align}
where $h(\cdot)$ is any harmonic function.

Substituting $\rho_i=\rho_0(\boldsymbol x)+\varrho_i\chi_{\Omega_0}$ into \eqref{eq6}, we have
\begin{align*}
	(\varrho_1-\varrho_2)\int_{\Omega_0}\big(\int_{B_R}\rho_1(\boldsymbol z)g_1(\boldsymbol z)~{\rm d}\boldsymbol z \big)h(\boldsymbol y)~{\rm d}\boldsymbol y=0.
\end{align*}
Because of the condition \eqref{addition2}, we get
\begin{align*}
	\varrho_1=\varrho_2,
\end{align*}
which yields $g_1(\boldsymbol x)=g_2(\boldsymbol x)$ and $f_1(\boldsymbol x)=f_2(\boldsymbol x)$.
\end{proof}

\begin{rema}
In fact, the condition \eqref{additioncon} is a special form for \eqref{addition2}.
There are other ways to achieve the non-zero condition for \eqref{addition2}.
\end{rema}

Besides the assumptions of density and internal sources in  Theorem \ref{t3} and Corollary \ref{coro1},
we also consider whether there are other circumstances in which a more general uniqueness result can be derived.
The following example illustrates a more general result.
\begin{exam}
	Let $(\rho_1, f_1, g_1)$ and $(\rho_2, f_2, g_2)$ be two sets of configurations and supported in $\Omega$, which satisfy
	\[(\rho_2, f_2, g_2)=(\rho_1+a, f_1+b, g_1+c)\]
	with $a(\boldsymbol x), b(\boldsymbol x), c(\boldsymbol x)\in L^\infty(\mathbb R^3)$ are nonnegative and supported in $\Omega$.
	Furthermore, suppose that $f_1(\boldsymbol x), g_1(\boldsymbol x)>0$.
	If
	\begin{align*}
		\Lambda_{\rho_1, f_1, g_1}(t, \boldsymbol x)=\Lambda_{\rho_2, f_2, g_2}(t, \boldsymbol x),~\quad (t, \boldsymbol x)\in \mathbb R_+\times \partial\Omega,
	\end{align*}
then
\[a(\boldsymbol x)=b(\boldsymbol x)=c(\boldsymbol x)=0.\]
\end{exam}
\begin{proof}
Assume that at least one of $a(\boldsymbol x), b(\boldsymbol x)$ and  $c(\boldsymbol x)$ is not zero.
Without losing of generality, we set $a\neq 0$, then we have
 \begin{align*}
	(\rho_1f_1-\rho_2f_2)(\boldsymbol x)=&(\rho_1f_1-(\rho_1f_1+af_1+b\rho_1+ab))(\boldsymbol x)=-(af_1+b\rho_1+ab)(\boldsymbol x)<0, \\
	(\rho_1g_1-\rho_2g_2)(\boldsymbol x)=&(\rho_1g_1-(\rho_1g_1+ag_1+c\rho_1+ac))(\boldsymbol x)=-(ag_1+c\rho_1+ac)(\boldsymbol x)<0,
\end{align*}
which yields
\begin{align}
\int_{\Omega}	(\rho_1f_1-\rho_2f_2)(\boldsymbol x)~{\rm d}\boldsymbol x<0, \label{e1}\\
\int_{\Omega}	(\rho_1g_1-\rho_2g_2)(\boldsymbol x)~{\rm d}\boldsymbol x<0. \label{e2}
\end{align}
It follows from \eqref{uniqueness} and \eqref{uniqueness1} that
\begin{align*}
	\int_{\Omega}(\rho_1f_1-\rho_2f_2)(\boldsymbol x)~{\rm d}\boldsymbol x=&0,\\
	\int_{\Omega}(\rho_1g_1-\rho_2g_2)(\boldsymbol x)~{\rm d}\boldsymbol x=&0,
\end{align*}
when taking $h(\boldsymbol x)=1$.
Therefore, the inequalities \eqref{e1} and \eqref{e2} are contradiction.
The proof is completed.
\end{proof}

It can be seen from the above example that a more general uniqueness result can be obtained if additional assumptions of density and sources are considered. The detail process will be shown in the following section.
%-------------------------------------------------------------------------------------------------------------------------------------------------------------------------------------------------

\section{Extension to  general results}\label{s4}
In this section, the previous assumptions of density and internal sources will be replaced by assuming some size relationships of density, internal sources and their coupling term, which implies  more general unique results.
\begin{lemm}\label{t5}
	Assume that  $(\rho_1, f_1, g_1)$ and $(\rho_2, f_2, g_2)$ are two sets of configurations and supported in $\Omega$.
	If
	\begin{align*}
	\Lambda_{\rho_1, f_1, g_1}(t, \boldsymbol x)=\Lambda_{\rho_2, f_2, g_2}(t, \boldsymbol x),~\quad (t, \boldsymbol x)\in \mathbb R_+\times \partial\Omega,
	\end{align*}
and  satisfies
	\begin{align}
			(\rho_1g_1)(\boldsymbol x)\leq (\rho_2g_2)(\boldsymbol x)~\quad\text{or}~\quad (\rho_1g_1)(\boldsymbol x)\geq (\rho_2g_2)(\boldsymbol x),~\quad \boldsymbol x\in \Omega.\label{auxi7}
	\end{align}
Then
\begin{align}\label{N6}
	\rho_1(\boldsymbol x)g_1(\boldsymbol x)= \rho_2(\boldsymbol x)g_2(\boldsymbol x).
\end{align}

In addition, if
\begin{align*}
		\int_{\mathbb R^3}(\rho_ig_i)(\boldsymbol x)~{\rm d}\boldsymbol x\neq 0, ~\quad i=1, 2,
	\end{align*}
we have
\begin{align}\label{N1}
	\int_{\mathbb R^3}(\rho_1-\rho_2)(\boldsymbol x)h(\boldsymbol x)~{\rm d}\boldsymbol x=0,
\end{align}
where $h(\boldsymbol x)$ is any harmonic function in $\mathbb R^3$.
\end{lemm}
\begin{proof}
Let $h(\boldsymbol x)=1$, then it follows from \eqref{uniqueness} and \eqref{uniqueness1} that
\begin{align}
\int_{B_R}(\rho_1f_1)(\boldsymbol y)~{\rm d}\boldsymbol y=&\int_{B_R}(\rho_2f_2)(\boldsymbol y)~{\rm d}\boldsymbol y,\label{auxilaryine}\\
\int_{B_R}(\rho_1g_1)(\boldsymbol y)~{\rm d}\boldsymbol y=&\int_{B_R}(\rho_2g_2)(\boldsymbol y)~{\rm d}\boldsymbol y. \label{auxilaryine2}
\end{align}
Given by the conditions \eqref{auxi7}, if
\[(\rho_1g_1-\rho_2g_2)(\boldsymbol x)>0 ~\quad  \text{or}~\quad (\rho_1g_1-\rho_2g_2)(\boldsymbol x)<0,\]  we have
\begin{align*}
\int_{B_R}(\rho_1g_1-\rho_2g_2)(\boldsymbol y)	~{\rm d}\boldsymbol y>0~\quad\text{or}\quad\int_{B_R}(\rho_1g_1-\rho_2g_2)(\boldsymbol y)	~{\rm d}\boldsymbol y<0.
\end{align*}
This contradiction with \eqref{auxilaryine2}.
Therefore, we get
\begin{align*}
	(\rho_1g_1)(\boldsymbol x)=(\rho_2g_2)(\boldsymbol x),
\end{align*}
 and imply
\begin{align}\label{N5}
	\int_{B_R}(\rho_1g_1)(\boldsymbol y)|\boldsymbol x-\boldsymbol y|^2~{\rm d}\boldsymbol y=\int_{B_R}(\rho_2g_2)(\boldsymbol y)|\boldsymbol x-\boldsymbol y|^2~{\rm d}\boldsymbol y.
\end{align}
Substituting \eqref{auxilaryine}-\eqref{N5} into \eqref{eq3}, we obtain
\begin{align*}
	\int_{B_R}(\rho_1-\rho_2)(\boldsymbol y)~g_0(|\boldsymbol x-\boldsymbol y|)~{\rm d}\boldsymbol y=0.
\end{align*}
Repeating the process of \eqref{uniqueness} and \eqref{uniqueness1} in Theorem \ref{t1}, we derive an orthogonal relation
\begin{align*}
	\int_{\mathbb R^3}(\rho_1-\rho_2)(\boldsymbol x)~h(\boldsymbol x)~{\rm d}\boldsymbol x=0
\end{align*}
for any harmonic function $h(\boldsymbol x)$ in $\mathbb R^3$.
\end{proof}

\begin{coro}
With the same assumptions of Lemma \ref{t5}.
If $\rho_i(\boldsymbol x), i=1, 2$ satisfies the either of the following conditions:
\begin{enumerate}[(i)]
	\item $(\rho_1-\rho_2)(\boldsymbol x)$ is a harmonic function in $\mathbb R^3$;
	\item $\rho_1(\boldsymbol x)\leq \rho_2(\boldsymbol x)$ or $\rho_1(\boldsymbol x)\geq \rho_2(\boldsymbol x),~\quad  \boldsymbol x\in \Omega$.
\end{enumerate}
Then  \[\rho_1(\boldsymbol x)=\rho_2(\boldsymbol x)~\quad \text{and}~\quad  g_1(\boldsymbol x)=g_2(\boldsymbol x).\]

Furthermore, suppose that
\begin{align}\label{N8}
	f_1(\boldsymbol x)\leq f_2(\boldsymbol x)~\quad \text{or} ~\quad f_1(\boldsymbol x)\geq f_2(\boldsymbol x),~\quad  \boldsymbol x\in \Omega.
\end{align}
Then \[f_1(\boldsymbol x)=f_2(\boldsymbol x).\]		
\end{coro}
\begin{proof}
For the first case, taking $(\rho_1-\rho_2)(\boldsymbol x)=h(\boldsymbol x)$ into \eqref{N1}, which implies
\begin{align*}
	\int_{\Omega}h^2(\boldsymbol x)~{\rm d}\boldsymbol x=0.
\end{align*}
Thus we can  obtain  $\rho_1(\boldsymbol x)=\rho_2(\boldsymbol x)$.

For the second case, substituting $h(\boldsymbol x)=1$ into \eqref{N1}, we get
\begin{align*}
	\int_{\Omega}(\rho_1-\rho_2)(\boldsymbol x)~{\rm d}\boldsymbol x=0.
\end{align*}
By using the conditions of $\rho_1(\boldsymbol x)$ and $\rho_2(\boldsymbol x)$, we deduce $\rho_1(\boldsymbol x)=\rho_2(\boldsymbol x)$.
It follows from \eqref{N6} that $g_1(\boldsymbol x)=g_2(\boldsymbol x)$.

Let $\rho_1(\boldsymbol x)=\rho_2(\boldsymbol x)=\rho(\boldsymbol x)$ and $h(\boldsymbol x)=1$, plugging them into \eqref{uniqueness}, we have
\begin{align*}
\int_\Omega \rho(\boldsymbol y)(f_1-f_2)	(\boldsymbol y)~{\rm d}\boldsymbol y=0.
\end{align*}
Therefore, given by \eqref{N8}, $f_1(\boldsymbol x)= f_2(\boldsymbol x)$ is proved.
\end{proof}

%\section*{Acknowledgment}
%The research of YG was  supported by NSFC grants 11871140, 12071065 and National Key R\&D Program of China 2020YFA0714102.
%The research
%of HL was supported  by NSFC/RGC Joint Research Scheme, N\_CityU101/21, ANR/RGC Joint Research Scheme, A-CityU203/19, and the Hong Kong RGC General Research Funds (projects 12302919, 12301420 and 11300821).

%------------------------------------------------------------------------------------------------------------------------------------------------------------------------------------------------

%\bibliographystyle{abbrv}
%%%%\bibliographystyle{elsarticle-num}
%%%%\bibliographystyle{UNRT}
%%%%\bibliographystyle{IEEEtran}
%%%%\bibliography{IEEEabrv,REFS}
%\bibliography{GLL}
%-------------------------------------------------------------------------------------------------------------------------------------------------------------------------------------------------

\begin{appendix}
\section{The well-posedness for the exterior boundary problem.}
We prove the well-posedness for the exterior boundary value problem
\begin{align}\label{ex}
\begin{cases}
	\Delta^2\hat u-\kappa^4\hat u=0,~\quad &{\rm in}~\mathbb R^3\backslash\overline\Omega,\\
		\hat u=\phi_1,~\quad \Delta \hat u=\phi_2,~\quad &{\rm on}~\partial\Omega
	\end{cases}
\end{align}
with the Sommerfeld radiation conditions
\begin{align}\label{S}
\underset{r\rightarrow\infty}{\lim}r \left(\partial_r \hat u(\kappa,  \boldsymbol x)-{\rm i}\kappa\hat u(\kappa, \boldsymbol  x)\right)=0	,~\quad
\underset{r\rightarrow\infty}{\lim}r\left(\partial_{r} (\Delta\hat u(\kappa, \boldsymbol x))-{\rm i}\kappa(\Delta\hat u(\kappa, \boldsymbol x))\right)=0.	
\end{align}
\begin{theo}\label{direct}
There exists a unique solution for \eqref{ex} with the Sommerfeld conditions \eqref{S}.
\end{theo}
\begin{proof}
Let $\tilde{\hat u}=\hat u_1-\hat u_2$, where $\hat u_1$ and $\hat u_2$ are solutions of \eqref{ex}--\eqref{S}, then it satisfies
\begin{align}\label{zm}
	\begin{cases}
	\Delta^2\tilde{\hat u}-\kappa^4\tilde{\hat u}=0,~\quad &{\rm in}~\mathbb R^3\backslash\overline\Omega,\\
		\tilde{\hat u}=0,~\quad \Delta \tilde{\hat u}=0,~\quad &{\rm on}~\partial\Omega
	\end{cases}
\end{align}
with the Sommerfeld conditions \eqref{S}.

 Denote that $\Omega_r=B_r\backslash\Omega$, where $B_r$ is a sphere of radius $r$ and center at the origin.
Let  $\Delta\tilde{\hat u}=-\kappa^2 w$,  then the first equation in \eqref{zm} turns to
\begin{align*}
\begin{cases}
\Delta \tilde{\hat u}+\kappa^2 w=0, ~\quad{\rm in}~\mathbb R^3\backslash\overline\Omega,\\
\Delta	w+\kappa^2\tilde{\hat u}=0, ~\quad{\rm in}~\mathbb R^3\backslash\overline\Omega.
\end{cases}
\end{align*}
Taking the inner product of the first equation with $\tilde{\hat u}$ and the second equation with $w$  over $\Omega_r$ and sum together, we compute
\begin{align}
	\int_{\Omega_r}	\Delta \tilde{\hat u}\bar{\tilde{\hat u}}+\kappa^2w\bar{\tilde{\hat u}}+\Delta w \bar w+\kappa^2\tilde{\hat u}\bar w~{\rm d}\boldsymbol x=0,\label{M1}\\
	\int_{\Omega_r}	\tilde{\hat u}\Delta\bar{\tilde{\hat u}}+\tilde{\hat u}\kappa^2\bar w+ w \Delta\bar w+w\kappa^2\bar{\tilde{\hat u}}~{\rm d}\boldsymbol x=0\label{M2}.
\end{align}
Subtracting \eqref{M2} with \eqref{M1} and  taking integration by parts, we have
\begin{align*}
0=&\int_{\Omega_r}	\Delta \tilde{\hat u}\bar{\tilde{\hat u}}+\kappa^2w\bar{\tilde{\hat u}}+\Delta w \bar w+\kappa^2\tilde{\hat u}\bar w~{\rm d}\boldsymbol x-\int_{\Omega_r}	\tilde{\hat u}\Delta\bar{\tilde{\hat u}}+\tilde{\hat u}\kappa^2\bar w+ w \Delta\bar w+w\kappa^2\bar{\tilde{\hat u}}~{\rm d}\boldsymbol x\\
=&\int_{\Omega_r}	\Delta \tilde{\hat u}\bar{\tilde{\hat u}}+\Delta w \bar w~{\rm d}\boldsymbol x-\int_{\Omega_r}	\tilde{\hat u}\Delta\bar{\tilde{\hat u}}+ w \Delta\bar w~{\rm d}\boldsymbol x\\
=&\int_{\partial B_r}	\partial_\nu \tilde{\hat u}\bar{\tilde{\hat u}}+\partial_\nu w \bar w~{\rm d}S-\int_{\partial\Omega}	\partial_\nu \tilde{\hat u}\bar{\tilde{\hat u}}+\partial_\nu w \bar w~{\rm d}S-\int_{\Omega_r} |\nabla\tilde{\hat u}|^2+|\nabla w|^2~{\rm d}\boldsymbol x\\
&-\int_{\partial B_r}	\tilde{\hat u}\partial_\nu\bar{\tilde{\hat u}}+ w \partial_\nu\bar w~{\rm d}S+\int_{\partial \Omega}	\tilde{\hat u}\partial_\nu\bar{\tilde{\hat u}}+ w \partial_\nu\bar w~{\rm d}S+\int_{\Omega_r} |\nabla\tilde{\hat u}|^2+|\nabla w|^2~{\rm d}\boldsymbol x\\
=&2{\rm i}{\rm Im}\int_{\partial B_r}	\partial_\nu \tilde{\hat u}\bar{\tilde{\hat u}}+\partial_\nu \bar w  w~{\rm d}S\\
=&2{\rm i}{\rm Im}\int_{\partial B_r}	{\rm i}\kappa |\tilde{\hat u}|^2+{\rm i}\kappa  |w|^2 ~{\rm d}S+\mathcal O(\frac{1}{r^2})(2{\rm i}{\rm Im}\int_{\partial B_r}	\bar{\tilde{\hat u}}+\bar w ~{\rm d}S)\\
=&2{\rm i}\kappa \int_{\partial B_r}	|\tilde{\hat u}|^2+  |w|^2 ~{\rm d}S+\mathcal O(\frac{1}{r^2})(2{\rm i}{\rm Im}\int_{\partial B_r}	\bar{\tilde{\hat u}}+\bar w ~{\rm d}S).
\end{align*}
Letting $r\rightarrow\infty$, we obtain the identity
\begin{align*}
	\underset{r\rightarrow\infty}{\lim}\int_{\partial B_r}	|\tilde{\hat u}|^2+  |w|^2 ~{\rm d}S=0.
\end{align*}
It follows from the Rellich's Lemma that $\tilde{\hat u}=0$ in $\mathbb R^3\backslash\overline\Omega$.

Next, we prove the existence of the solution.
The proof depends on the boundary integral equation.

Assume that $u_H$ and $u_M$ are the solution of $\Delta u+\kappa^2 u	=0$ and $\Delta u-\kappa^2 u=0$, respectively, then
$\hat u=u_H+u_M$ is a solution to $\Delta^2\hat u-\kappa^4\hat u=0$.

Let $G_H(|\boldsymbol x-\boldsymbol y|)$ be the Green function of the Helmholtz equation
\begin{align*}
\Delta u+\kappa^2 u	=0,
\end{align*}
and $G_M(|\boldsymbol x-\boldsymbol y|)$  be the Green function of the Modified Helmholtz equation
\begin{align*}
\Delta u-\kappa^2 u=0.	
\end{align*}
Given integrable functions $\varphi$ and $\psi$, we define the single-layer potential  and the double-layer potential
\begin{align*}
v_s(\boldsymbol x):=&\int_{\partial\Omega}	G_H(|\boldsymbol x-\boldsymbol y|)\varphi(\boldsymbol y)~{\rm d}\boldsymbol y+\int_{\partial\Omega}G_M(|\boldsymbol x-\boldsymbol y|)\psi(\boldsymbol y)~{\rm d}\boldsymbol y,~\quad \boldsymbol x\in \mathbb R^3\backslash\partial\Omega,\\
v_d(\boldsymbol x):=&\int_{\partial\Omega}	\frac{\partial G_H(|\boldsymbol x-\boldsymbol y|)}{\partial \nu(\boldsymbol y)}\varphi(\boldsymbol y)~{\rm d}\boldsymbol y+\int_{\partial\Omega}\frac{\partial G_M(|\boldsymbol x-\boldsymbol y|)}{\partial \nu(\boldsymbol y)}\psi(\boldsymbol y)~{\rm d}\boldsymbol y,~\quad \boldsymbol x\in \mathbb R^3\backslash\partial\Omega.
\end{align*}
Taking a solution of a combination of the double- and single-layer potential
\begin{align}\label{zl}
\hat u(\boldsymbol x)=v_d(\boldsymbol x)-{\rm i}\gamma v_s(\boldsymbol x),~\quad \boldsymbol x\in \mathbb R^3\backslash\partial\Omega,
\end{align}
where $\gamma$ is a nonzero constant.
Since the jump relations on $\partial\Omega$, we see that  $\hat u$ given by \eqref{zl} in $\mathbb R^3\backslash\partial\Omega$ solves the exterior problem \eqref{ex} provided $(\varphi, \psi)$ is a solution of the integral equation
\begin{align*}
\begin{cases}
	(-{\rm i}\gamma S^H+K^H+\frac{1}{2})\varphi+(-{\rm i}\gamma S^M+K^M+\frac{1}{2})\psi=\phi_1,\\
	(-{\rm i}\gamma S^H+K^H+\frac{1}{2})\varphi-(-{\rm i}\gamma S^M+K^M+\frac{1}{2})\psi=-\frac{\phi_2}{\kappa^2},
\end{cases}
\end{align*}
where
\begin{align*}
(S^H\varphi)(\boldsymbol x):=&	\int_{\partial\Omega}	G_H(|\boldsymbol x-\boldsymbol y|)\varphi(\boldsymbol y)~{\rm d}\boldsymbol y,~\quad \boldsymbol x\in \partial\Omega,\\
(S^M\psi)(\boldsymbol x):=&	\int_{\partial\Omega}	G_M(|\boldsymbol x-\boldsymbol y|)\psi(\boldsymbol y)~{\rm d}\boldsymbol y,~\quad \boldsymbol x\in \partial\Omega,\\
(K^H\varphi)(\boldsymbol x):=&\int_{\partial\Omega}	\frac{\partial G_H(|\boldsymbol x-\boldsymbol y|)}{\partial \nu(\boldsymbol y)}\varphi(\boldsymbol y)~{\rm d}\boldsymbol y,~\quad \boldsymbol x\in \partial\Omega,\\
(K^M\psi)(\boldsymbol x):=&\int_{\partial\Omega}	\frac{\partial G_M(|\boldsymbol x-\boldsymbol y|)}{\partial \nu(\boldsymbol y)}\psi(\boldsymbol y)~{\rm d}\boldsymbol y,~\quad \boldsymbol x\in \partial\Omega.
\end{align*}
By direct calculation, we have
\begin{align}\label{zx}
\begin{cases}
	(-{\rm i}\gamma S^H+K^H+\frac{1}{2})\varphi=\frac{1}{2}(\phi_1-\frac{\phi_2}{\kappa^2}),\\
	(-{\rm i}\gamma S^M+K^M+\frac{1}{2})\psi=\frac{1}{2}(\phi_1+\frac{\phi_2}{\kappa^2}).
\end{cases}	
\end{align}
 Therefore, the existence of $(\varphi, \psi)$ to \eqref{zx} can be established by the Riesz-Fredholm theory with the compactness of $S^H, S^M, K^H$ and $K^M$ (see \cite{Co2019,colton2013}).
Then the representation \eqref{zl} is a solution for the exterior boundary value problem \eqref{ex}.
The proof is completed.
  \end{proof}		
\end{appendix}

\end{document}